\documentclass[10pt]{amsart}
\usepackage[utf8]{inputenc}
\usepackage{amsmath}
\usepackage{amssymb, bm}
\usepackage{enumerate}
\usepackage[all]{xy}

\makeatletter
\@namedef{subjclassname@2010}{
\textup{2010} Mathematics Subject Classification}
\makeatother

\newtheorem{Thm}{Theorem}[section]
\newtheorem{Prp}[Thm]{Proposition}
\newtheorem{Cor}[Thm]{Corollary}
\newtheorem{Lma}[Thm]{Lemma}
\theoremstyle{definition}
\newtheorem{Def}[Thm]{Definition}
\newtheorem{Rk}[Thm]{Remark}

\numberwithin{equation}{section}

\newcommand{\Id}{\operatorname{Id}}
\newcommand{\ID}{\mathbf{Id}}
\newcommand{\sgn}{\operatorname{sgn}}
\newcommand{\Pf}{\mathfrak{Pf}}

\newcommand{\ort}{\mathrm{O}}
\newcommand{\N}{\mathbb N}

\newcommand{\PGO}{\mathbf{PGO}}
\newcommand{\pgo}{\mathrm{PGO}}
\newcommand{\gl}{\mathrm{GL}}
\newcommand{\GO}{\mathbf{GO}}
\newcommand{\go}{\mathrm{GO}}

\newcommand{\End}{\mathrm{End}}
\newcommand{\ad}{\mathrm{ad}}
\newcommand{\Comp}{\mathfrak{Comp}}

\newcommand{\Aut}{\mathbf{Aut}}
\newcommand{\aut}{\mathrm{Aut}}
\newcommand{\Inn}{\mathbf{Inn}}
\newcommand{\Alg}{\text{-}\mathfrak{Alg}}
\newcommand{\Hom}{\operatorname{Hom}}
\newcommand{\Set}{\mathfrak{Set}}
\newcommand{\AUT}{\mathfrak{Aut}}
\newcommand{\TRI}{\mathfrak{Tri}}
\newcommand{\Grp}{\mathfrak{Grp}}
\newcommand{\G}{\mathbf{G}}

\begin{document}
\title[Composition Algebras and Algebraic Groups]{Composition Algebras and Outer Automorphisms of Algebraic Groups}
\author[S. Alsaody]{Seidon Alsaody}
\address{Uppsala University\\Dept.\ of Mathematics\\P.O.\ Box 480\\751 06 Uppsala\\Sweden}
\email{seidon.alsaody@math.uu.se}

\begin{abstract} In this note, we establish an equivalence of categories between the category of all eight-dimensional
composition algebras with any given quadratic form $n$ over a field $k$ of characteristic not two, and a category arising from an
action of the projective similarity group of $n$ on certain pairs of automorphisms of the group scheme $\PGO^+(n)$ defined over
$k$. This extends results recently obtained in the same direction for symmetric composition algebras. We also derive known results
on composition algebras from our equivalence.
\end{abstract}

\maketitle

\section{Introduction}
The problem of classifying finite-dimensional composition algebras up to isomorphism is a long open one which has attracted much
attention, as in \cite{EM}, \cite{EP}, \cite{P}, \cite{Ca} and \cite{A4}. Some classes of these algebras, such as the symmetric
ones, are
quite well understood. In general, however, the problem is far from being solved. Finite-dimensional composition algebras
necessarily have dimension 1, 2, 4 or 8, and it is the eight-dimensional case which is the most widely open and the one on which
we shall focus here. Part of the difficulty arises from the appearance of the triality phenomenon in the isomorphism criteria.

In the recent paper \cite{CKT}, the authors established a correspondence between eight-dimensional symmetric composition
algebras
with quadratic form $n$ over a field $k$ and outer automorphisms of order three of the affine group scheme $\PGO^+(n)$, called
trialitarian automorphisms. They
further
showed that this induces a bijection between isomorphism classes of symmetric composition algebras on the one hand, and conjugacy
classes of trialitarian automorphisms on the other. This was further developed in \cite{CEKT}.

Our aim is to extend this approach. Namely, building on the results of \cite{CKT} and \cite{CEKT}, we relate not necessarily
symmetric composition algebras of dimension eight with
quadratic form $n$ to certain pairs of outer automorphisms of $\PGO^+(n)$, and prove that isomorphisms of such algebras correspond
bijectively to simultaneous conjugation by inner automorphisms of $\PGO(n)$. We are then also able to derive some previously known
isomorphism conditions from this description, which sheds new light on the classification problem of composition algebras.

\section{Preliminaries}
Throughout the paper $k$ denotes an arbitrary field of characteristic not two.

\subsection{Composition Algebras}
Let $V$ be a finite-dimensional vector space. A \emph{quadratic form} on $V$ is a map $n:V\to k$ satisfying $n(\alpha
x)=\alpha^2n(x)$
for each $\alpha\in k$ and $x\in V$, and such that the map $b_n:V\times V\to k$ defined by 
\[b_n(x,y)=n(x+y)-n(x)-n(y)\]
is bilinear. The pair $(V,n)$ is then called a \emph{quadratic space}. For each subset $U\subseteq V$, we write $U^\bot$ for
the orthogonal complement of $U$
with respect to $b_n$. The form $n$ is said to be
\emph{non-degenerate} if $V^\bot=\{0\}$.

\begin{Def} An \emph{algebra over $k$} or \emph{$k$-algebra} is a $k$-vector space with a bilinear multiplication. A
\emph{composition
algebra} over $k$ is a $k$-algebra endowed with a non-degenerate multiplicative quadratic form $n$.
\end{Def}

We denote the operators of left and right multiplication by an element $a$ in an algebra $A$ by $L_a^A$ and $R_a^A$, respectively.
By definition, the datum of a composition algebra is a triple $(C,\cdot,n)$ where $C$ denotes the vector space, $\cdot$ the
multiplication,
and $n$ the quadratic form. Below we shall subsume the multiplication, and often the quadratic form as well, in the notation,
thus 
writing $C$ or $(C,n)$ for $(C,\cdot,n)$.

\begin{Rk}\label{R: generalized} In \cite{CKT} the term \emph{normalized composition} is used for the structure of what is here
called a
composition algebra, and the term \emph{composition} is used for a more general structure, which we will call \emph{generalized
composition algebras}. In this sense, a generalized composition algebra over $k$ is a $k$-algebra $A$ endowed with a
non-degenerate quadratic form $n$ for which there exists $\lambda\in k^*$, called the \emph{multiplier} of $A$, with
$n(xy)=\lambda n(x)n(y)$ for all $x,y\in A$. (Thus a generalized composition algebra with multiplier $\lambda$ is a composition
algebra precisely when $\lambda=1$.) Our focus and terminology agrees with \cite{CEKT}, \cite{KMRT} and large parts of the
literature. At one occasion we shall use generalized composition algebras as an
intermediate step in a proof.
\end{Rk}

A finite-dimensional composition algebra is said to be a \emph{Hurwitz algebra} if it is unital. The following result is due to \cite{Ka}.

\begin{Lma}\label{L: Kaplansky} Let $(C,n)$ be a finite-dimensional composition algebra. Then there exist a Hurwitz algebra $H$
and $f,g\in \ort(n)$
such that $(C,n)=(H_{f,g},n)$.
\end{Lma}

For any algebra $A$ and any pair $(f,g)$ of linear operators on $A$, the \emph{isotope} $A_{f,g}$ is the algebra with underlying space $A$ and
multiplication
\[x\cdot y=f(x)g(y),\]
where juxtaposition denotes the multiplication of $A$. A
definition of the orthogonal group $\ort(n)$ is given in Section 1.4 below.

Indeed, if $(C,n)$ is a finite-dimensional composition algebra, then the non-degene\-racy of $n$ implies the existence of some
$c\in C$ with $n(c)\neq 0$. Setting $e=n(c)^{-1}c^2$ and $H=C_{(R_e^C)^{-1},(L_e^C)^{-1}}$ we see that $(H,n)$ is a Hurwitz
algebra with unity $e$, and $C=H_{f,g}$ with $f=R_e^C$ and $g=L_e^C$. These belong to $\ort(n)$ as $n(e)=1$ and $n$
is multiplicative.

Another class of composition algebras of particular interest is that of symmetric composition algebras.

\begin{Def} Let $(C,n)$ be a composition algebra. The bilinear form $b_n$ is called \emph{associative} if, for all
$x,y,z\in C$,
\[b_n(xy,z)=b_n(x,yz).\]
The algebra is called \emph{symmetric} if the bilinear form $b_n$ is associative.
\end{Def}

We will use the same terminology for generalized composition algebras. 

Any Hurwitz algebra $H$ is endowed with a canonical involution $i$, defined by fixing
the unity and acting as $x\mapsto -x$ on its orthogonal complement. This involution is an isometry and the \emph{para-Hurwitz
algebra} $H_{i,i}$ is symmetric. Thus the
following lemma is an immediate consequence of the 
previous lemma.

\begin{Lma}\label{L: symmetric} Let $(C,n)$ be a finite-dimensional composition algebra. Then there exist a symmetric
composition algebra $S$
and $f,g\in \ort(n)$
such that $(C,n)=(S_{f,g},n)$.
\end{Lma}

When dealing with isotopes of algebras and their isomorphisms, the following easy result is useful. Its proof is a
straightforward verification which we leave out.

\begin{Lma}\label{L: juggling} Let $A$ and $B$ be arbitrary $k$-algebras and let $f$ and $g$ be invertible linear operators on
$A$. If $h:A\to B$ is an isomorphism, then $h f h^{-1}$ and $h g h^{-1}$ are invertible linear operators on $B$,
and
\[h:A_{f,g}\to B_{h f h^{-1},h g h^{-1}}\]
is an isomorphism.
\end{Lma}

A natural question to ask is when a finite-dimensional quadratic space $(V,n)$ admits the structure of a composition algebra. It
is known (see e.g.\ \cite[(33.18)]{KMRT}) that this is the case if and only if for some $r\in\{0,1,2,3\}$, $\dim V=d=2^r$
and $n$ is
an $r$-fold Pfister form. We will write $n\in\Pf_r(k)$ to denote that $n$ is an $r$-fold Pfister form over $k$. We also write
$\Comp_d$ for the category of
all $d$-dimensional composition algebras over $k$, where the morphisms are all non-zero
algebra homomorphisms. These are known to be isometries, hence injective by the non-degeneracy of the quadratic forms, whence they are
isomorphisms
for dimension reasons. For each $n\in\Pf_r(k)$ we
write $\Comp(n)$ for the full subcategory of all composition algebras with quadratic form $n$. As in
\cite{CKT} and \cite{CEKT}, we shall consider $\Comp(n)$ for a fixed
quadratic form $n\in\Pf_3(k)$, which is arbitrarily chosen. It may not \emph{a priori} be clear if this is useful in understanding
the category $\Comp_8$, in the sense that composition algebras with different quadratic forms may still be isomorphic. The following known result
shows
that this is not a problem.

\begin{Prp} Let $n,n'\in\Pf_r(k)$ for some $r\in\{0,1,2,3\}$.
\begin{enumerate}
\item If $n$ and $n'$ are not isometric, $C\in\Comp(n)$ and $C'\in\Comp(n')$, then $C$ is not isomorphic to $C'$.
\item If $n$ and $n'$ are isometric, then every $C'\in\Comp(n')$ is isomorphic to some $C\in\Comp(n)$.
\end{enumerate}
\end{Prp}

\begin{proof} The first item is due to \cite{Pet}. For the second, Lemma \ref{L: Kaplansky} provides a Hurwitz algebra
$H'\in\Comp(n')$ and $f',g'\in \ort(n')$ such that $C'=H'_{f',g'}$. The same lemma shows that if $\Comp(n)$ is not empty, then it
contains a Hurwitz algebra $H$. Since $n$ and $n'$ are isometric, we conclude with \cite[(33.19)]{KMRT} that there exists
an isomorphism $h:H'\to H$, which by the above is an isometry. Then the maps $f=h f' h^{-1}$ and $g=h g' h^{-1}$ belong to
$\ort(n)$,
whence $H_{f,g}\in\Comp(n)$, and from $h:H'\to H$ being an isomorphism it follows that
\[h: H'_{f',g'}\to H_{f,g}\]
is an isomorphism.
\end{proof}

For each isometry class $N\subseteq\Pf_3(k)$, we write $\Comp(N)$ for the full subcategory of $\Comp_8$ of all composition
algebras whose quadratic form belongs to $N$. Then the above implies the coproduct decomposition
\[\Comp_8=\coprod_{N\subseteq\Pf_3(k)} \Comp(N),\]
and moreover, for each such class
$N$ and each $n\in N$, the full subcategory $\Comp(n)$ is dense in $\Comp(N)$. Thus the classification problem of $\Comp_8$ may
be treated by fixing one quadratic form at a time, and it then suffices to consider one quadratic form from each isometry class.

Finally, given any
algebra $A$ over $k$ and $\lambda \in k$, the \emph{scalar multiple} $\lambda A$ is defined to be equal to $A$ as a vector space,
with
multiplication
\[x\cdot y=\lambda xy.\]
Using isotopes, this can be formulated by saying that $\lambda A=A_{\lambda\Id,\Id}=A_{\Id,\lambda\Id}$.

\begin{Lma}\label{L: scalars} Let $A$ be a $k$-algebra and $\lambda\in k^*$.
\begin{enumerate}
\item The map $\lambda\Id: \lambda A\to A$ is an isomorphism of algebras.
\item If $A\in \Comp(n)$, then $\lambda A\in \Comp(n)$ if and only if $\lambda=\pm1$.
\end{enumerate}
\end{Lma}

The proof is a straightforward verification.

\subsection{Affine Group Schemes}
We shall use the functorial approach to algebraic groups, which is developed in \cite{DG}, \cite{W} and \cite{KMRT}. We
denote by $k\Alg$ the category of all unital commutative
associative algebras over $k$, with algebra homomorphisms as morphisms.\footnote{Despite this notation, we shall not assume
associativity, commutativity or unitality when we use the word \emph{algebra} in general.} An \emph{affine scheme} is a functor
$\mathbf F: k\Alg\to\Set$ which is \emph{representable}, i.e.\ isomorphic to
$\Hom(A_0,-)$ for some $A_0\in k\Alg$. An \emph{affine group scheme} over
$k$
is a functor $\mathbf G: k\Alg\to\Grp$ which is representable in the sense that its composition with the forgetful functor into
$\Set$ is. A \emph{(normal) subgroup functor} of a functor $\mathbf G:k\Alg\to\Grp$ is a functor $\mathbf H: k\Alg\to\Grp$ such
that $\mathbf H(A)$ is a (normal) subgroup of $\mathbf G(A)$ for each $A\in k\Alg$.

Given a functor $\G:k\Alg\to\Grp$, we write $\Aut(\G)(k)$ for the group of automorphisms of $\G$ defined over $k$. Here, we
understand an \emph{automorphism of $\G$ defined over $k$}, or briefly an \emph{automorphism of $\G$} to be a natural
transformation $\bm\eta$ from $\G$ to itself such that for each
$A\in k\Alg$, the map $\bm\eta_A:\G(A)\to \G(A)$ is a group automorphism.
Each automorphism of $\G$ thus in particular induces an automorphism of the group $G:=\G(k)$. Inner automorphisms of $G$ can
conversely be lifted to automorphisms of $\G$. Indeed, if $g\in G$ and $A\in k\Alg$, then the image under $\G$ of the
inclusion $\iota_A: k\to A, \alpha\mapsto\alpha1$, is a group homomorphism, and the map
\[(\bm{\kappa}_g)_A: \G(A)\to \G(A), h\mapsto \G(\iota_A)(g)h\G(\iota_A)(g)^{-1}\]
is a group automorphism of $\G(A)$. The following fact is clear.

\begin{Lma}\label{L: lift} Let $\G: k\Alg \to \Grp$ be a functor, and let $g\in G$. Then the map $\bm{\kappa}_g: \G \to \G$ given,
for
each $A\in k\Alg$, by $(\bm{\kappa}_g)_A$, is an automorphism of $\G$. If $\mathbf H$ is a normal
subgroup functor of $\G$, then
$\bm{\kappa}_g$ restricts to an automorphism of $\mathbf H$, and for
any $r\in \N$,
the group $G$ acts on $\left(\Aut(\mathbf H)(k)\right)^r$ by
\[g\cdot (\bm\alpha_1,\ldots,\bm\alpha_r)=(\bm\kappa_g \bm\alpha_1\bm\kappa_g^{-1},\bm\kappa_g \bm\alpha_r\bm\kappa_g^{-1}).\]
\end{Lma}

Automorphisms of the form $\bm\kappa_g$ for some $g\in G$ of $\G$ are called \emph{inner}, and automorphisms which are
not inner are called \emph{outer}. We write $\Inn(\G)(k)$ for the subgroup of $\Aut(\G)(k)$ consisting of all inner automorphisms.
If $\mathbf H$ is a normal subgroup functor of $\G$, we call
automorphism of $\mathbf H$ \emph{weakly inner} (with respect to $\G$) if it is equal to $\bm\kappa_g$ for some $g\in G$, and
\emph{strongly outer} otherwise.

\subsection{Groups and Group Schemes of Quadratic Forms}
To keep our presentation reasonably self-contained, we will in this section give an introduction to
those groups and group schemes that will be needed in the present paper. Our summary is based on \cite{KMRT} and \cite{CKT}, the
former of which contains an extensive account of the theory of groups and group schemes related to quadratic (and other) forms.

Let $(V,n)$ be a non-degenerate quadratic space. A \emph{similarity} of $(V,n)$ is a linear map $f:V\to V$ such that there exists
$\mu(f)\in k^*$ with
\[n(f(x))=\mu(f)n(x)\]
for all $x\in V$. Similarities form a group denoted by $\go(n)$. For each $f\in \go(n)$, the scalar $\mu(f)$ is
called the \emph{multiplier} of $f$, and the assignment $f\mapsto \mu(f)$ defines a group homomorphism $\mu: \go(n)\to k^*$. The
kernel
of $\mu$ is the \emph{orthogonal group} $\ort(n)$, consisting of all isometries of $V$ with respect to $n$. On the other hand,
there is a monomorphism in the other direction, i.e.\ from $k^*$ to $\go(n)$, given by $\lambda\mapsto \lambda\Id$. Abusing
notation, we write $k^*$ for its image, which is a central, hence normal, subgroup, and define $\pgo(n)=\go(n)/k^*$. In the
category of groups, we therefore have the two exact sequences
\[1\to \ort(n) \to  \go(n) \xrightarrow{\mu} k^* \to 1\]
and
\[1\to k^* \to  \go(n) \to \pgo(n) \to 1.\]

On the level of group schemes, we will mainly be concerned with $\PGO(n)$. To define it, we use the notion of the adjoint
involution of the strictly non-degenerate quadratic form $n$. This is an anti-automorphism $\ad_n$ of $\End_kV$ defined
by
\[b_n(f(x),y)=(x,\ad_n(f)(y)).\]
Thus $(\End_kV,\ad_n)$ is an algebra with involution, and we then define $\PGO(n)$ as the automorphism group scheme of this
algebra, i.e.\ for each $A\in k\Alg$,
\[\PGO(n)(A)=\left\{\psi\in\aut_A(A\otimes\End_kV)|\psi(\ad_n)_A=(\ad_n)_A\psi\right\},\]
where $(\ad_n)_A$ is defined on $A\otimes\End_kV$ by extending $\ad_n$.

From \cite[\S 23]{KMRT}, we know that $\PGO(n)\simeq \GO(n)/\G_m$, where $\GO(n)$ is the group scheme of similarities of $n$, with
$\GO(n)(k)=\go(n)$, and $\G_m$ is multiplicative group scheme, with $\G_m(k)=k^*$.
Note however that this is a quotient of group schemes, and does not imply that $\PGO(n)(A)\simeq\GO(N)(A)/\G_m(A)$ 
for each $A$ in $k\Alg$. This is however true for $A=k$. Indeed, the Skolem--Noether theorem implies that
each $k$-algebra automorphism of $\End_kV$ is inner, and it is then straightforward to check that an inner automorphism $g\mapsto
fgf^{-1}$ of $\End_kV$ commutes with $\ad_n$ if and only if $f\in\go(n)$. This defines a surjective group homomorphism from
$\go(n)$ to $\PGO(n)(k)$, and as $\End_kV$ is central, the kernel of this homomorphism is $k\Id$. Thus we get an isomorphism of
groups
\[\begin{array}{ll}\pgo(n)\to \PGO(n)(k)& [f]\mapsto \left(g\mapsto fgf^{-1}\right).  \end{array}\]

\begin{Rk}\label{R: identification} Following custom, we will henceforth identify $\PGO(n)(k)$ with $\pgo(n)$ in view of the above
isomorphism. Thus for each
$h\in\go(n)$ we identify the inner automorphism $f\mapsto hfh^{-1}$ with the coset $[h]$ of $h$ in $\pgo(n)$.
\end{Rk}

We finally write $\PGO^+(n)$ for the connected component of the identity of $\PGO(n)$. Under the identification above, its group of rational points corresponds
to $\pgo^+(n)=\go^+(n)/k^*$. To define $\go^+(n)$, let $C(V,n)$
be the Clifford algebra of the form $n$, with even part denoted by $C_0(V,n)$. The canonical involution of $C(V,n)$ which induces
the identity map on $V$, is
denoted by $\sigma$. Then it is known (see e.g.\ \cite[\S 13]{KMRT}) that each $f\in\go(n)$ induces a automorphism $C(f)$
of $C_0(V,n)$,
and that the restriction of $C(f)$ to the centre of $C_0(V,n)$ either is the identity or has order two. The similarity $f\in
\go(n)$
is called proper if $C(f)$
induces the identity on the centre, and the set $\go(n)^+$ of all proper similarity forms a normal subgroup of index two of
$\go(n)$. We also write
$\ort^+(n)=\go^+(n)\cap\ort(n)$ for the group of proper isometries of $n$.

\subsection{Triality for Unital and Symmetric Composition Algebras} The Principle of Triality, originating from E.\ Cartan
\cite{C}, is the following statement.

\begin{Prp} Let $n\in\Pf_3(k)$ and let $H\in\Comp(n)$ be a Hurwitz algebra. For each $h\in \ort^+(n)$ there exist
$h_1,h_2\in\ort^+(n)$ such that for all $x,y\in H$,
\[h(xy)=h_1(x)h_2(y),\]
and the pair $(h_1,h_2)$ is unique up to multiplication by $\lambda$ in the first argument and $\lambda^{-1}$ in the second
for some $\lambda\in k^*$.
\end{Prp}

A proof and an elaborate discussion can be found in \cite{SV}.

In \cite{CKT} and \cite{CEKT}, triality was discussed using symmetric, rather than unital, composition algebras. We recall some of
their main results.

\begin{Prp}\label{P: CKT} Let $S$ be a finite-dimensional symmetric generalized composition algebra over $k$, with
quadratic form $n\in\Pf_3(k)$. 
\begin{enumerate}[(i)]
\item For each $h\in \go^+(n)$ there exist $h_1,h_2\in \go^+(n)$ such that for any $x,y\in S$,
\[h(xy)=h_1(x)h_2(y).\]
The pair $([h_1],[h_2])\in\pgo^+(n)$ is determined uniquely by $h$.
\item For each $r\in\{1,2\}$, the map $\rho_r^S:[h]\mapsto[h_r]$ is an outer automorphism of $\pgo^+(n)$ of order
three, and induces an automorphism $\bm\rho_r^S$ of
$\PGO^+(n)$ defined over $k$. Moreover, $\bm\rho_1^S\bm\rho_2^S=\ID$.
\item The assignment $S\mapsto \bm\rho_1^S$ defines a bijection between eight-dimensional symmetric generalized
composition algebras of dimension eight up to scalar multiples,
and outer automorphisms of
$\PGO^+(n)$ of order three defined over $k$. 
\item For any symmetric generalized composition algebra $T$ with quadratic form $n$ and any $h\in\go(n)$, 
$\kappa_{[h]}\rho_1^S\kappa_{[h]}^{-1}=\rho_1^T$ if and only if there exists $\lambda\in
k^*$ such that $\lambda h$ is
an isomorphism from $S$ to $T$.
\end{enumerate}
\end{Prp}

From \cite[\S 35]{KMRT} we have the following description of $\Aut\left(\PGO^+(n)\right)(k)$.

\begin{Prp}\label{P: KMRT} Let $n\in\Pf_3(k)$ and let $S\in \Comp(n)$ be symmetric. Then 
\[\Aut\left(\PGO^+(n)\right)(k)=\left\{\bm\kappa_{[g]}\left(\bm\rho_1^S\right)^r|g\in\go(n)\wedge r\in\{0,\pm1\}
\right\}.\]
Moreover, the group $\Aut\left(\PGO^+(n)\right)(k)/\Inn\left(\PGO^+(n)\right)(k)$ is generated by the cosets of
$\bm\rho_1^S$ and $\bm\kappa_{[g]}$ for any $g\in\go(n)\setminus\go^+(n)$, and is isomorphic to $S_3$.
\end{Prp}

The next lemma shows that automorphisms of $\PGO^+(n)$ are determined by their action on rational points.

\begin{Lma}\label{L: independent} Let $n\in\Pf_3(k)$ and $\bm\alpha\in\Aut\left(\PGO^+(n)\right)(k)$. If $\bm\alpha$ induces the
identity on $\pgo^+(n)$, then $\bm\alpha$ is the identity in $\Aut\left(\PGO^+(n)\right)(k)$.
\end{Lma}

\begin{proof} Let $S\in\Comp(n)$ be symmetric. By Proposition \ref{P: KMRT} and Remark \ref{R: identification}, we have
$\bm\alpha=\bm\kappa_{[h]}\left(\bm\rho_1^S\right)^r$ for some $h\in\go(n)$ and $r\in\{0,\pm1\}$. If $\bm\alpha$ induces the
identity on $\pgo^+(n)$, we have $\kappa_{[h]}\left(\rho^S\right)^r=\Id$. Since $\rho^S$ is of order three, this implies
that
\[\kappa_{[h^2]}=\left(\rho_1^S\right)^r,\]
and since the left hand side is an inner automorphism of $\pgo^+(n)$ and $\left(\rho_1^S\right)^{\pm1}$ is not by Proposition
\ref{P: CKT},
we have $r=0$
and thus $\bm\alpha=\bm\kappa_{[h]}$. Since the centralizer of $\pgo^+(n)$ in $\pgo(n)$ is trivial we moreover have $[h]=1$, whence
$\bm\kappa_{[h]}$ is the identity. Altogether $\bm\alpha$ is the identity, as desired.
\end{proof}

\subsection{Group Action Groupoids}
By a \emph{groupoid} we understand any category where all morphisms are isomorphisms. Let $G$ be a group acting on a set
$X$.\footnote{Unless otherwise stated, all group actions are assumed to be from the left.} This gives rise to a groupoid $_GX$ as
follows. The object set of $_GX$ is $X$, and for each $x,y\in X$, the set of morphisms from $x$ to $y$ is
\[_GX(x,y)=\{(x,y,g)\in X\times X\times G|g\cdot x=y\}.\]
When the objects $x$ and $y$ are clear from context, we will denote the morphism $(x,y,g)$ simply by $g$. Group action groupoids
are used to construct \emph{descriptions} of certain groupoids in the sense of \cite{D}, whereby a description of a
groupoid $\mathfrak C$ is an equivalence of categories from a group action groupoid to $\mathfrak C$. Given a description of a
groupoid, classifying it up to isomorphism is then transferred to solving the normal form problem for the group action, i.e.\
constructing a cross-section for its orbits. Our approach in this paper is similar, as it consists of constructing an equivalence
of categories from a
groupoid of algebras to a group action groupoid.

\section{Triality for Eight-Dimensional Composition Algebras}
Throughout this section, we fix a quadratic form $n\in\Pf_3(k)$.

\begin{Lma}\label{L: Triality} Let  $C\in\Comp(n)$. Then for each $h\in \go^+(n)$ there exists
a pair $(h_1,h_2)\in \go^+(n)^2$ such that for each $x,y\in C$,
 \[h(xy)=h_1(x)h_2(y),\]
and the pair $([h_1],[h_2])\in\pgo^+(n)^2$ is unique.
\end{Lma}

The maps $h_1$ and $h_2$ are called \emph{triality components of $h$ with respect to $C$}.

\begin{proof} If $C$ is symmetric, then the statement follows from Proposition \ref{P: CKT}. In general there exists by Lemma
\ref{L: symmetric} a symmetric $S\in\Comp(n)$ and $f,g\in \ort(n)$ such that $C=S_{f,g}$. Denote the multiplication in $S$ by
juxtaposition and that in $C$ by $\cdot$. Then if $h_1'$ and $h_2'$ are triality components of $h$ with respect to $S$, then
\begin{equation}\label{triality}
h(x\cdot y)=h(f(x)g(y))=h_1'f(x)h_2'g(y)=f^{-1}h_1'f(x)\cdot g^{-1}h_2'g(y),
\end{equation}
whence $h_1=f^{-1}h_1' f$ and $h_2=g^{-1}h_2' g$ are triality components of $h$ with respect to $C$. The uniqueness of
$([h_1],[h_2])$ follows from the uniqueness of $([h_1'],[h_2'])$.
\end{proof}

Note that \eqref{triality} does not use the symmetry of $S$, and arguing as in the above proof, we obtain the
following.

\begin{Lma}\label{L: relation} Let $B\in\Comp(n)$ and $f,g\in \ort(n)$. Then $C=B_{f,g}$ satisfies
\[(\rho_1^B,\rho_2^B)=(\kappa_{[f]}\rho_1^C,\kappa_{[g]}\rho_2^C).\]
\end{Lma}

Lemma \ref{L: Triality} defines, for any $C\in\Comp(n)$, two maps
\[\rho_1^C,\rho_2^C:\pgo^+(n)\to\pgo^+(n)\]
by $\rho_r^C([h])=[h_r]$ for each $r\in\{1,2\}$.
These in fact define automorphisms of affine group schemes, as the next proposition shows.

\begin{Prp}\label{P: rho} For each $C\in\Comp(n)$ and each $r\in\{1,2\}$, the map $\rho_r^C$ is a strongly outer automorphism of
$\pgo^+(n)$, and induces an automorphism $\bm\rho_r^C$ of $\PGO^+(n)$ defined over $k$.
\end{Prp}

\begin{proof} By Lemma \ref{L: symmetric}, there exist a symmetric
composition algebra $S$ and $f,g\in \ort(n)$ such
that $C=S_{f,g}$, and by Lemma \ref{L: relation},
\begin{equation}\label{rho}
(\rho_1^C,\rho_2^C)=(\kappa_{[f]}^{-1}\rho_1^S,\kappa_{[g]}^{-1}\rho_2^S).
\end{equation}
Proposition \ref{P: CKT}(ii) then implies that $\rho_1^C$ is an automorphism of $\pgo^+(n)$. From the same result we
further know that $\rho_1^S$ is an outer automorphism of $\pgo^+(n)$, and
it is strongly outer as its square is not inner. If $\rho_1^C=\kappa_{[h]}$ for some $h\in \go(n)$, then
\[\rho_1^S=\kappa_{[fh]},\]
contradicting the fact that $\rho_1^S$ is strongly outer. From Proposition \ref{P: CKT} together with Lemma \ref{L: lift} we
deduce that $\kappa_{[f]}^{-1}\rho_1^S$ induces an automorphism $\bm\rho_1^C=\bm\kappa_{[f]}^{-1}\bm\rho_1^S$
of
$\PGO^+(n)$ defined over $k$. Since $\bm\rho_1^C$ acts as $\rho_1^C$ on rational points, Lemma \ref{L:
independent} implies that it is independent of the choice of $S$, $f$ and $g$. The case of $\rho_2^C$ is treated
analogously, and the proof is complete.
\end{proof}

We will prove in the next section that the triality components of a composition algebra of dimension eight essentially carry all
the information about the composition algebra. At this point, we will prove that the triality components detect the property of
being symmetric, in the following sense.

\begin{Lma}\label{L: simultaneous} Assume that $C,D\in\Comp(n)$ and $h\in\go(n)$ satisfy
\[\rho_r^C=\kappa_{[h]}\rho_r^D\kappa_{[h]}^{-1}\]
for each $r\in\{1,2\}$. If $D$ is symmetric, then so is $C$.
\end{Lma}

\begin{proof} There exist by Lemma \ref{L: symmetric} a symmetric $S\in\Comp(n)$ and $f,g\in \ort(n)$ such that
$C=S_{f,g}$. By Lemma 5.2 of \cite{CKT}, $C$ is symmetric if (and only if) $f,g\in\ort^+(n)$ and satisfy
\begin{equation}\label{desired}
\begin{array}{lll} \prod_{m=0}^2 \left(\rho_1^S\right)^m([f])=1 & \text{and} & \left(\rho_1^S\right)^2\left([f]^{-1}\right)=[g].
  \end{array}
\end{equation}
To prove that $f\in\ort^+(n)$, recall from Lemma \ref{L: relation} that
$\kappa_{[h]}\rho_1^D\kappa_{[h]}^{-1}=\rho_1^C=\kappa_{[f]}^{-1}\rho_1^S$, and therefore 
\[\bm\kappa_{[h]}\bm\rho_1^D \bm\kappa_{[h]}^{-1}=\bm\kappa_{[f]}^{-1}\bm\rho_1^S\]
in the
group $\Aut\left(\PGO^+(n)\right)(k)$. In 
\[\Aut\left(\PGO^+(n)\right)(k)/\Inn\left(\PGO^+(n)\right)(k)\simeq S_3,\]
the coset of the left hand side has
order three,
while if $f\notin\ort^+(n)$,
the coset of the right hand side is a product of an element of order two with an element of order three, which by the structure of
$S_3$ has
order two. Thus $f\in\ort^+(n)$, and a similar argument gives $g\in\ort^+(n)$.

For the first equality in \eqref{desired}, it follows from $\rho_1^C=\kappa_{[f]}^{-1}\rho_1^S$ that
\[\left(\kappa_{[f]}^{-1}\rho_1^S\right)^3=\left(\kappa_{[h]}\rho_1^D\kappa_{[h]}^{-1}\right)^3.\]
Expanding and using the fact that $\rho_1^S$ and $\rho_1^D$ are homomorphisms, we get
\[\kappa_{[f']}\left(\rho_1^S\right)^3=\kappa_{[h]}\left(\rho_1^D\right)^3\kappa_{[h]}^{-1}\]
with $[f']$ being the inverse of $\prod_{m=0}^2 \left(\rho_1^C\right)^m([f])$. Since $D$ and $S$ are symmetric, $\rho_r^D$ and
$\rho_r^S$ have order three for each $r$. Thus $\kappa_{[f']}=\Id$, and since the centralizer of $\pgo^+(n)$ in $\pgo(n)$ is
trivial, this proves that $[f']$ is trivial, whence the first equality follows.

For the second we have $\rho_2^C=\kappa_{[g]}^{-1}\rho_2^S$, and
$\rho_2^T=\left(\rho_1^T\right)^2$ whenever $T\in\Comp(n)$ is symmetric. Thus
\[\kappa_{[g]}^{-1}\left(\rho_1^S\right)^2=\kappa_{[g]}^{-1}
\rho_2^S=\rho_2^C=\kappa_{[h]}\rho_2^D\kappa_{[h]}^{-1}=\kappa_{[h]}
\left(\rho_1^D\right)^2\kappa_{[h]}^{-1}=\left(\rho_1^C\right)^2.\]
The rightmost expression equals $\left(\kappa_{[f]}^{-1}\rho_1^S\right)^2=\kappa_{[f'']}^{-1}\left(\rho_1^S\right)^2$ with
$[f'']=\rho_1^S\left([f]\right)[f]$. Thus
\[\kappa_{[g]}^{-1}=\kappa_{[f'']}^{-1}\]
and triviality of the centralizer of $\pgo^+(n)$ gives $[g]=[f'']$, which by the first equality in \eqref{desired} equals
$\left(\rho_1^S\right)^2\left([f]^{-1}\right)$. This proves the second equality, and the proof is complete.
\end{proof}

\section{Pairs of Outer Automorphisms}
As in the previous section we fix an arbitrary $n\in\Pf_3(k)$. In this section we prove that $\Comp(n)$ is equivalent to a full
subcategory of a group action groupoid.

\subsection{Constructing the Functor}
By Lemma \ref{L: lift} we know that $\pgo(n)$ acts on 
\[\Aut\left(\PGO^+(n)\right)(k)\times\Aut\left(\PGO^+(n)\right)(k)\]
by
\[[h]\cdot(\bm\alpha_1,\bm\alpha_2)=\left(\bm\kappa_{[h]}\bm\alpha_1\bm\kappa_{[h]}^{-1},\bm\kappa_{[h]}\bm\alpha_2\bm\kappa_{[h]}
^{ -1}\right).\]
In the fashion described in Section 2.5, this group action gives rise to the group action groupoid
\[\AUT(n)=\phantom{}_{\pgo(n)}\left(\Aut_k(\PGO^+(n))(k)\times\Aut_k(\PGO^+(n))(k)\right).\]
As a step toward showing that $\Comp(n)$ is equivalent to a full subcategory of $\AUT(n)$, we begin by proving that isomorphisms
in $\Comp(n)$ correspond to isomorphisms in $\AUT(n)$.

\begin{Prp} Let $C,D\in\Comp(n)$. Then there is a bijection
\[\Comp(n)(C,D)\to\AUT(n)\left((\bm\rho_1^C,\bm\rho_2^C),(\bm\rho_1^D,\bm\rho_2^D)\right)\]
given by $h\mapsto [h]$.
\end{Prp}

Our proof of the well-definedness of the map will closely generalize that of Proposition \ref{P: CKT}(iv), which is given in
\cite{CKT}. Our proof of
surjectivity instead transfers the problem to the corresponding problem for symmetric algebras, to which that proposition applies.

\begin{proof} In view of Lemma \ref{L: independent}, proving that the map is well-defined amounts to showing that for each
isomorphism $h:C\to D$
we have
\begin{equation}\label{pair}
\left(\kappa_{[h]}\rho_1^C\kappa_{[h]}^{-1},\kappa_{[h]}\rho_1^C\kappa_{[h]}^{-1}\right)=\left(\rho_1^D,
\rho_2^D\right).
\end{equation}
Assume therefore that $h: C\to D$ is an isomorphism. Then $h\in\ort(n)$, whence $[h]\in\pgo(n)$. Let $j\in \go^+(n)$ and let
$(j_1,j_2)$ be a pair of triality
components of $j$ with respect to $D$. Denoting multiplication in $C$ by juxtaposition and in $D$ by $\cdot$ we have, for all
$x,y\in C$,
\[h^{-1}j h(xy)=h^{-1}j(h(x)\cdot h(y))=h^{-1}
(j_1h(x)\cdot j_2h(y))=h^{-1}j_1h(x)h^{-1}j_2h(y).\]
Thus
\[\left(h^{-1}j_1h,h^{-1}j_2h\right)\]
is a pair of triality components for $h^{-1}j h$ with respect to $C$, and the uniqueness statement of Lemma \ref{L: Triality}
implies
that
\[\rho_r^C([h^{-1}j h])=[h^{-1}]\rho_r^D[j][h]\]
which, since $j$ was chosen arbitrarily, implies \eqref{pair}. Thus the map is well-defined. It is injective since if
$h,j: C\to D$ are isomorphisms with $[h]=[j]$, then for some
$\lambda\in k^*$ we have $j=\lambda h$. Thus for all $x,y\in C$
\[\lambda h(xy)=j(xy)=j(x)\cdot j(y)=\lambda h(x)\cdot\lambda h(y)=\lambda^2h(xy),\]
whence $\lambda=1$ and thus $j=h$.

It remains to show surjectivity. Let $h\in \go(n)$ be such that
\[\bm\rho_r^D=\bm\kappa_{[h]}\bm\rho_r^C \bm\kappa_{[h]}^{-1},\]
whence
\[\rho_r^D=\kappa_{[h]}\rho_r^C\kappa_{[h]}^{-1},\]
for each $r\in\{1,2\}$. By Lemma \ref{L: symmetric} there exist symmetric $S,T\in\Comp(n)$ and $f,g,f',g'\in\ort(n)$ such that
$C=S_{f,g}$ and $D=T_{f',g'}$. Then by Lemma \ref{L: relation},
\[\kappa_{[h]}\kappa_{[f]}^{-1}\rho_1^S\kappa_{[h]}^{-1}=\kappa_{[h]}\rho_1^C\kappa_{[h]}^{-1}=\rho_1^D=\kappa_{[f']}^{-1}\rho_1^T,\]
which is equivalent to
\[\kappa_{[h]}\rho_1^S\kappa_{[h]}^{-1}=\kappa_{[h fh^{-1}{f'}^{-1}]}\rho_1^T.\]
Similarly one obtains
\[\kappa_{[h]}\rho_2^S\kappa_{[h]}^{-1}=\kappa_{[h gh^{-1}{g'}^{-1}]}\rho_2^T.\]
Now $f''=f'h f^{-1}h^{-1}$ and $g''=g'h g^{-1}h^{-1}$ belong to $\ort(n)$, whence $B=T_{f'',g''}$ is in $\Comp(n)$, and
from the above two lines and Lemma \ref{L: relation} we deduce
\[\kappa_{[h]}\rho_r^S\kappa_{[h]}^{-1}=\rho_r^B\]
for each $i$. Since $S$ is symmetric, Lemma \ref{L: simultaneous} implies that $B$ is symmetric, and then with Proposition \ref{P:
CKT}(iv) we conclude that for some $\lambda\in k^*$, $\lambda h$ is an isomorphism from $S$ to $B$. Since $C=S_{f,g}$ and
\[D=T_{f',g'}=B_{h f h^{-1},h g h^{-1}}=B_{\lambda h f (\lambda h)^{-1},\lambda h g (\lambda h)^{-1}},\]
we conclude from Lemma \ref{L: juggling} that $\lambda h$ is an isomorphism from $C$ to $D$. Since $[h]=[\lambda h]$, and
$h\in\go(n)$ was chosen arbitrarily, this proves surjectivity.
\end{proof}

Time is now ripe to prove our main result.

\begin{Thm}\label{T: main} The map
\[\mathfrak F: \Comp(n)\to\AUT(n)\]
defined on objects by
\[C\mapsto \left(\bm\rho_1^C,\bm\rho_2^C\right)\]
and on morphisms by $h\mapsto [h]$, is a full and faithful functor, which is injective on objects up to sign.
\end{Thm}

By \emph{injective up to sign}, we mean that if $\mathfrak F(C)=\mathfrak F(D)$, then $D$ is a scalar multiple of $C$ with scalar
$\pm1$.

\begin{proof} The map $\mathfrak F$ is well-defined on objects by Proposition \ref{P: rho}. It is well-defined on morphisms by
the above lemma, and clearly maps identities to identities and respects composition of morphisms. Thus $\mathfrak F$ is a
functor, which by the lemma above is full and faithful. It remains to be shown that $\mathfrak F$ is injective up to sign.

Assume that $C,D\in\Comp(n)$ satisfy $\bm\rho_r^C=\bm\rho_r^D$, and thence $\rho_r^C=\rho_r^D$, for each $r$. Then
$C=S_{f,g}$
and $D=T_{f',g'}$ with $S,T$
symmetric and $f,g,f',g'\in\ort(n)$, and by Lemma \ref{L: relation},
\[\left(\kappa_{[f'f^{-1}]}\rho_1^S,\kappa_{[g'g^{-1}]}\rho_2^S\right)=\left(\rho_1^T,\rho_2^T\right).\]
But the left hand side is equal to $(\rho_1^B,\rho_2^B)$ with $B=S_{f{f'}^{-1},g{g'}^{-1}}$. Thus $B$ is symmetric by Lemma \ref{L:
simultaneous}, and then by Theorem 5.8 in \cite{CKT}, we have
$B=\lambda T$ for some $\lambda\in k^*$. Then $\lambda=\pm 1$ by Lemma \ref{L: scalars}. Therefore,
\[C=S_{f,g}=\left(S_{f{f'}^{-1},g{g'}^{-1}}\right)_{f',g'}=B_{f',g'}=\lambda T_{f',g'}=\lambda D,\]
proving injectivity up to sign. This completes the proof.
\end{proof}

\subsection{Trialitarian Pairs}
The aim of this section is to describe the image of the functor $\mathfrak F$. We begin by the following observation. If
$(\bm\alpha_1,\bm\alpha_2)$ is an object in $\AUT(n)$, then there exists a symmetric $S\in\Comp(n)$ such that
$\bm\alpha_1=\bm\kappa_{[f]}\bm\rho_1^S$
for some $f\in\go(n)$, and then $\bm\alpha_2=\bm\kappa_{[g]}\left(\bm\rho_2^S\right)^r$ for some $g\in\go(n)$ and $r\in\{0,\pm1\}$.
If
$r=1$
and there exist $f',g'\in \ort(n)$ with $[f']=[f]$ and $[g']=[g]$, then by Lemma \ref{L: relation},
\[(\bm\alpha_1,\bm\alpha_2)=\mathfrak F\left(S_{{f'}^{-1},{g'}^{-1}}\right).\]
We will, roughly speaking, show that the requirement on $r$ is essential, but that, up to isomorphism, that on $f$ and $g$ is
not. To formalize and prove the latter statement, we need the following result.

\begin{Lma}\label{L: normalization} Let $F,G\in \go(n)$ and let $S\in\Comp(n)$ be a para-Hurwitz algebra. Then there exist
$f,g\in \ort(n)$ and a symmetric $T\in \Comp(n)$ such that
\begin{equation}\label{generalized}
\left(\kappa_{[F]}\rho_1^S,\kappa_{[G]}\rho_2^S\right)\simeq\left(\kappa_{[f]}\rho_1^T,\kappa_{[g]}\rho_2^T\right)
\end{equation}
in $\AUT(n)$.
\end{Lma}

In the proof we need to pass to symmetric generalized composition algebras. To treat these
we shall refer to a couple of results proved in \cite{CKT}, where, as previously remarked, these algebras are called
\emph{symmetric compositions}, while symmetric composition algebras are named \emph{normalized symmetric compositions}.

\begin{proof} By definition of a para-Hurwitz algebra we have $S=H_{i,i}$ where $H$ is a Hurwitz algebra with canonical involution
$i$. Denoting the unity of $H$ by $e$, and setting $a=F^{-1}(e)$ and $b=G^{-1}(e)$ we have
\[\mu\left(F^{-1}\right)=n\left(F^{-1}(e)\right)=n(a),\]
and since $\mu$ is a group homomorphism,
\[\mu\left(FiR_a^H\right)=n(a)^{-1}n(i)n(a)=1,\]
whereby $FiR_a^H\in\ort(n)$, and similarly one gets $GiL_b^H\in \ort(n)$. The algebra 
\[H'=H_{R_a^H,L_b^H}\]
is in fact a unital
generalized composition algebra with norm $n$, unity $e'=(ab)^{-1}$ and multiplier $\lambda=n(ab)$. One can check that the map
\[\begin{array}{ll}
  i':H'\to H',& x\mapsto b_n(x,e')e'-x
  \end{array}
\]
is an anti-automorphism on $H'$ as well as an isometry, and further that $S'=H'_{i',i'}$ is a symmetric generalized composition
algebra. (The proof of the latter statement is straightforward and consists of manipulations analogous to those used to prove
that a para-Hurwitz algebra is a symmetric composition algebra.)

Altogether,
\[S'=S_{iR_a^Hi',iL_b^Hi'}\]
and since $S$ and $S'$ are symmetric and $iR_a^Hi',iL_b^Hi'\in\go(n)$, Lemma 5.2 from \cite{CKT} applies to give
\[\begin{array}{lll}\rho_1^{S'}=\kappa_{\left[iR_a^Hi'\right]}^{-1}\rho_1^S& \text{and}&
\rho_2^{S'}=\kappa_{\left[iL_b^Hi'\right]}^{-1}\rho_2^S.\end{array}\]
On the other hand, Proposition 3.6 from \cite{CKT} states that each symmetric generalized composition algebra is isomorphic to a
(unique) symmetric composition algebra. Thus there exists
a symmetric composition algebra $T$ and an isomorphism $h: T\to S'$, with $h\in\go(n)$. Therefore, by Proposition \ref{P: CKT},
\[\begin{array}{lll}\rho_1^{S'}=\kappa_{[h]}\rho_1^{T}\kappa_{[h]}^{-1}& \text{and}&
\rho_2^{S'}=\kappa_{[h]}\rho_2^{T}\kappa_{[h]}^{-1}.\end{array}\]
Equating the two above expressions of $\rho_1^{S'}$ we get
\[\kappa_{[F]}\rho_1^S=\kappa_{\left[FiR_a^Hi'\right]}\kappa_{[h]}\rho_1^T\kappa_{[h]}^{-1}=\kappa_{[h]}\kappa_{\left[h^{-1}
FiR_a^Hi'h\right]}\rho_1^T\kappa_{[h]}^{-1},\]
and $h^{-1}FiR_a^Hi'h\in\ort(n)$ since, as we have already concluded, $FiR_a^H\in \ort(n)$ and $i'\in\ort(n)$. Likewise,
\[\kappa_{[G]}\rho_2^S=\kappa_{[h]}\kappa_{\left[h^{-1}GiL_b^Hi'h\right]}\rho_2^T\kappa_{[h]}^{-1},\]
with $h^{-1}GiL_b^Hi'h\in\ort(n)$. This proves the existence $f,g\in\ort(n)$ such that \eqref{generalized} holds, and the
proof is complete.
\end{proof}

Next we will construct a full subcategory of $\AUT(n)$ in which the image of $\mathfrak F$ is dense, and to which, by Theorem \ref{T: main}, $\Comp(n)$ is
therefore equivalent.

To construct this subcategory, let $\Inn^*\left(\PGO^+(n)\right)(k)$ denote the subgroup of $\Aut\left(\PGO^+(n)\right)(k)$
consisting of all weakly
inner
automorphisms (with respect to $\PGO(n)$). From Proposition \ref{P: KMRT} we deduce that the set
\[\Delta(n)=\Aut\left(\PGO^+(n)\right)(k)/\Inn^*\left(\PGO^+(n)\right)(k)\]
of all left cosets of this subgroup consists of three elements. Denoting the quotient projection by $\pi$ we observe that any
outer
automorphism $\bm\alpha\in\Aut\left(\PGO^+(n)\right)(k)$ of order three satisfies
\[\Delta(n)=\pi(\{\ID,\bm\alpha,\bm\alpha^2\}).\]
Generalizing this, we call $(\bm\alpha_1,\bm\alpha_2)\in\AUT(n)$ a \emph{trialitarian pair} if
\[\pi(\{\ID,\bm\alpha_1,\bm\alpha_2\})=\Delta(n).\]
As an automorphism $\bm\alpha$ of $\PGO^+(n)$ is weakly inner if and only if $\pi(\bm\alpha)=\pi(\ID)$, this is equivalent
to
requiring
$\bm\alpha_1$ and $\bm\alpha_2$ to be strongly outer and have different images under $\pi$. We denote the set of trialitarian
pairs
by
$\TRI(n)$.

\begin{Rk} Set $\Omega(n)=\Aut\left(\PGO^+(n)\right)(k)\setminus\Inn^*\left(\PGO^+(n)\right)(k)$, the set of all strongly outer
automorphisms of
$\PGO^+(n)$.
Consider the diagram
\[\xymatrix@1{& \Omega(n)\ar[d]^{\pi} \phantom{\: ,}\\ \Omega(n) \ar[r]_{\pi}& \Delta(n) \: ,}\]
the pullback of which (in the category of sets) is $\Omega(n)\times_{\Delta(n)}\Omega(n)$ with the corresponding projection maps.
Then we in fact have
\[\TRI(n)=\Omega(n)\times\Omega(n)\setminus\Omega(n)\times_{\Delta(n)}\Omega(n).\]
\end{Rk}

Denoting the full subcategory of $\AUT(n)$ with object set $\TRI(n)$ by $\TRI(n)$ as well, we can describe the image of $\mathfrak
F$ up to isomorphism as follows.

\begin{Prp} The image of $\Comp(n)$ under $\mathfrak F$ is dense in $\TRI(n)$.
\end{Prp}

\begin{proof} Let $C\in\Comp(n)$. Then $C=S_{f,g}$ for some symmetric $S\in\Comp(n)$ and $f,g\in\ort(n)$. Thus by Lemma \ref{L:
relation}, we have
\[\left(\bm\rho_1^C,\bm\rho_2^C\right)=\left(\bm\kappa_{[f]}^{-1}\bm\rho_1^S,\bm\kappa_{[g]}^{-1}\bm\rho_2^S\right),\]
which belongs to $\TRI(n)$ since $\bm\rho_1^S$ is an outer automorphism of order three and $\bm\rho_2^S=\left(\bm\rho_1^S\right)^2$. Thus
$\mathfrak{F}\left(\Comp(n)\right)\subseteq\TRI(n)$.

To prove denseness, assume that $(\bm\alpha_1,\bm\alpha_2)\in\TRI(n)$. Then for any para-Hurwitz algebra $S\in \Comp(n)$ there exist
$F,G\in\go(n)$
with
 \[(\bm\alpha_1,\bm\alpha_2)=\left(\bm\kappa_{[F]}^{-1}\bm\rho_1^S,\bm\kappa_{[G]}^{-1}\bm\rho_2^S\right),\]
and then Lemma \ref{L: normalization} provides a symmetric $T\in\Comp(n)$ and $f,g\in\ort(n)$ such that
 \[(\bm\alpha_1,\bm\alpha_2)=\left(\bm\kappa_{[f]}^{-1}\bm\rho_1^T,\bm\kappa_{[g]}^{-1}\bm\rho_2^T\right),\]
But then $C=T_{f^{-1},g^{-1}}\in \Comp(n)$, and, by Lemma \ref{L: relation}, $(\bm\alpha_1,\bm\alpha_2)\simeq\mathfrak F(C)$. This
completes
the proof.
\end{proof}

We have thus achieved our advertised goal, as we have proved the following.

\begin{Cor} The functor $\mathfrak F:\Comp(n)\to\AUT(n)$ induces an equivalence of categories $\Comp(n)\to\TRI(n)$.
\end{Cor}

\section{Previous Results Revisited}
In this section we will revisit known results about composition algebras, and express them in terms of the triality pairs of these
algebras.

\subsection{Symmetric Composition Algebras}
In order to precisely express how our approach generalizes that of \cite{CKT} and \cite{CEKT}, we begin by expressing the
structural results obtained there in the current framework. To this end,
consider the set of all
trialitarian automorphisms of $\PGO^+(n)$, i.e.\ outer automorphisms of $\PGO^+(n)$ of order three. The group $\pgo(n)$ acts on
this set by conjugation, viz.
\[[h]\cdot\bm\alpha=\bm\kappa_{[h]}\bm\alpha \bm\kappa_{[h]}^{-1},\]
and we denote the group action groupoid arising from this action by $\TRI^*(n)$. The following lemma is easy to check.

\begin{Lma} The map $\mathfrak{G}: \TRI^*(n)\to\AUT(n)$, defined on objects by $\bm\alpha\mapsto(\bm\alpha,\bm\alpha^2)$, and on
morphisms by
$[h]\mapsto[h]$, is a full and faithful functor, which is injective on objects. Moreover, the image of $\mathfrak{G}$ is a
full subcategory of $\TRI(n)$.
\end{Lma}

Thus $\mathfrak{G}$ induces an isomorphism of categories 
\[\mathfrak{G}':\TRI^*(n)\to\mathfrak G\left(\TRI^*(n)\right)\subseteq \TRI(n).\]
The
structural results from \cite{CKT} and \cite{CEKT} can now be reformulated as follows.

\begin{Prp} The functor $\mathfrak F:\Comp(n)\to\AUT(n)$ induces an equivalence of categories
$\mathfrak{F}':\Comp^S(n)\to\mathfrak{G}(\TRI^*(n))$.
\end{Prp}

Here, $\Comp^S(n)$ is the full subcategory of $\Comp(n)$ whose objects are all symmetric composition algebras in $\Comp(n)$.

\begin{proof} The induced functor is well-defined since if $S\in\Comp^S(n)$, then $\bm\rho_1^S$ is a trialitarian automorphism
with
square $\bm\rho_2^S$. It is full by fullness of $\Comp^S(n)$ and $\mathfrak{F}$, and faithful by faithfulness of the latter. If
$\bm\alpha\in \TRI^*(n)$, then from \cite{CKT} we know that $\bm\alpha=\bm\rho_1^S$ for some symmetric generalized composition
algebra
$S$, and that $\bm\alpha\simeq \bm\rho_1^T$ in $\TRI^*(n)$ for the unique symmetric composition algebra $T\simeq S$.
\end{proof}

In other words, composing $\mathfrak F'$ with the inverse of $\mathfrak{G}'$, one obtains an equivalence between the category of
symmetric composition algebras with quadratic form $n$ and the groupoid arising from  the action of $\pgo(n)$ on the set of all
trialitarian automorphisms of $\PGO^+(n)$ by conjugation by weakly inner automorphisms.

\subsection{The Double Sign}
The double sign was defined for finite-dimensional real division algebras in \cite{DD}, and the topic has been implicitly treated
for composition algebras over arbitrary fields of characteristic not two in e.g.\ \cite{EP}. Recall that for a composition algebra
$C$ and an element $a\in C$ with $n(a)\neq 0$, the left and right multiplication operators $L_a^C$ and $R_a^C$ are similarities.

\begin{Lma} Let $C$ be a finite-dimensional composition algebra over $k$, and let $a,b\in C$ be anisotropic. Then
$L_a^C$ is a proper similarity if and only if $L_b^C$ is, and $R_a^C$ is proper if and only if $R_b^C$ is.
\end{Lma}

An element of a quadratic space is called anisotropic if its is not in the kernel of the quadratic form. The proof is essentially
due to \cite{EP}.

\begin{proof} There exists a Hurwitz algebra $H$ and $f,g\in\go(n)$ such that $C=H_{f,g}$. Then $L_c^C=L_{f(c)}^Hg$ and
$R_c^C=R_{g(c)}^Hf$ for all $c\in C$. Now left and right multiplications by anisotropic elements in
Hurwitz algebras are proper similarities. Thus $L_c^C$ is proper if and only if $g$ is, and $R_c^C$ is proper if and only if $f$
is. This completes the proof.
\end{proof}

We define the \emph{sign} $\sgn (h)$ of a similarity $h$ as $+1$ if $h$ is proper, and $-1$ otherwise. The above lemma guarantees
that the following notion is well-defined.

\begin{Def} The \emph{double sign} of a finite-dimensional composition algebra $C$ is the pair $(\sgn(L_c^C),\sgn(R_c^C))$ for any
$c\in C$ with $n(c)\neq 0$.
\end{Def}

As noted in \cite{EP}, the double sign is the pair $(\det(L_c^C),\det(R_c^C))$ for any $c\in C$ with $n(c)=1$. Recall that such
elements
exist in any composition algebra. It is easily seen that isomorphic composition algebras have the same double sign. Thus
\[\Comp(n)=\coprod_{(r,s)\in\{\pm1\}^2} \Comp^{rs}(n),\]
where $\Comp^{rs}(n)$ is the full subcategory of $\Comp(n)$ consisting of all algebras with double sign $(r,s)$.

We can now prove that  the double sign can be inferred from the triality pair of the algebra.

\begin{Prp} Let $n\in\Pf_3(k)$ and $C\in\Comp(n)$. Then the double sign of $C$ is
$((-1)^{o_2},(-1)^{o_1})$, where $o_r$ is the order of the coset of $\bm\rho_r^C$ in the quotient group
$\Aut(\PGO^+(n))(k)/\Inn(\PGO^+(n))(k)$.
\end{Prp}

\begin{proof} We have $C=H_{f,g}$ for some Hurwitz algebra $H\in\Comp(n)$ and $f,g\in\ort(n)$. Since the canonical
involution $i$ on $H$ is not proper, we have 
\[(f,g)=\left(i^{p_1}f',i^{p_2}g'\right)\]
 for some $p_1,p_2\in\{0,1\}$ and $f',g'\in \ort^+(n)$.
Then by Lemma
\ref{L: relation}, for any $r\in\{1,2\}$, the coset of $\bm\rho_r^C$ equals the coset of $\bm\kappa_{[j]}^{p_r}\bm\rho_r^H$. If
$p_r=1$,
then the coset of $\bm\rho_r^C$ coincides with the coset of $\bm\rho_r^S$ for the symmetric composition algebra $S=H_{i,i}$, which
has order three. If
$p_r=0$, then the coset of $\bm\rho_r^C$ coincides with the coset of $\bm\kappa_{[j]}\bm\rho_r^S$, which has order two, since it
is product of an element of order two with one of order three. In both cases $o_i=p_i+2$. Now the double sign of
$H$ is $(+1,+1)$, which implies that the double sign of $C$ is $((-1)^{p_2},(-1)^{p_1})$. This completes the proof.
\end{proof}

\subsection{Isomorphism Conditions}
The category $\Comp(n)$ contains a unique isomorphism class $\mathfrak H(n)$ of Hurwitz algebras. In view of Lemmata \ref{L:
Kaplansky} and \ref{L: juggling}, one may fix a Hurwitz algebra $H\in\mathfrak H(n)$ and study the full subcategory of $\Comp(n)$
consisting of all
orthogonal isotopes of $H$, which is dense and hence equivalent to $\Comp(n)$. This approach is pursued in
\cite{Ca}, \cite{CDD} and \cite{A2} (for real division composition algebras) and in \cite{Da} (for general isotopes of Hurwitz
algebras over arbitrary fields).

We will recall the isomorphism conditions given in these papers, and give a new proof of these using our approach above. To begin
with, the following result is proved in \cite{Da}.

\begin{Lma} Let $(H,n)$ be a Hurwitz algebra and $f,g,f',g'\in\gl(A)$. If $h: H_{f,g}\to H_{f',g'}$ is an isomorphism,
then $h\in\go^+(n)$.
\end{Lma}

Moreover, an isomorphism condition is given in \cite{Da} for isotopes of Hurwitz algebras. For orthogonal isotopes of
eight-dimensional Hurwitz algebras, it implies the following.

\begin{Prp} Let $n\in\Pf_3(k)$ and let $H\in\Comp(n)$ be a Hurwitz algebra, $f,g,f',g'\in\ort(n)$, and $h\in\ort^+(n)$. Then $h:
H_{f,g}\to
H_{f',g'}$ is an isomorphism, if and only if
\begin{equation}\label{erik}
\left([f'],[g']\right)=\left([h_1][f][h]^{-1},[h_2][g][h]^{-1}\right),
\end{equation}
where $(h_1,h_2)$ is a pair of triality components of $h$ with respect to $H$.
\end{Prp}

A similar result is proven in \cite{Ca} for the case $k=\mathbb R$ and $n=n_E$, the standard Euclidean norm. We will now
show that this result in fact follows by applying Theorem \ref{T: main} above.

\begin{proof} Set $C=H_{f,g}$ and $D=H_{f',g'}$. By Theorem \ref{T: main}, $h:C\to D$ is an isomorphism if and only if
\[\left(\rho_1^D,\rho_2^D\right)=\left(\kappa_{[h]}\rho_1^C\kappa_{[h]}^{-1},\kappa_{[h]}\rho_2^C\kappa_{[h]}^{-1}\right).\]
By Lemma \ref{L: relation}, the statement $\rho_1^D=\kappa_{[h]}\rho_1^C\kappa_{[h]}^{-1}$ is equivalent to
\begin{equation}\label{equivalent}
\kappa_{[f']}^{-1}\rho_1^H=\kappa_{[h]}\kappa_{[f]}^{-1}\rho_1^H\kappa_{[h]}^{-1}.
\end{equation}
Now
\[\rho_1^H\kappa_{[h]}^{-1}=\kappa_{\rho_1^H([h])}^{-1}\rho_1^H\]
since $\rho_1^H$ is a group homomorphism and $[h]\in\pgo^+(n)$, and therefore \eqref{equivalent} is equivalent to
\[\kappa_{[f']}^{-1}=\kappa_{[h][f]^{-1}\rho_1^H([h])^{-1}},\]
which, since the centralizer of $\pgo^+(n)$ in $\pgo(n)$ is trivial, is in turn equivalent to
\[[f']=\rho_1^H([h])[f][h]^{-1}.\]
By an analogous argument, the statement $\rho_2^D=\kappa_{[h]}\rho_2^C\kappa_{[h]}^{-1}$ is equivalent to
\[[g']=\rho_2^H([h])[g][h]^{-1},\]
and thus altogether $h:C\to D$ is an isomorphism if and only if \eqref{erik} holds, as desired.
\end{proof}

\bibliographystyle{amsplain}
\bibliography{references}
\end{document}